\documentclass[a4paper,12pt]{article}
\usepackage{test}

\usepackage{hyperref}
\hypersetup{colorlinks,citecolor=blue,urlcolor=magenta}
\usepackage{doi}

\usepackage[
style=numeric-comp,
sorting=nyt,
maxcitenames=2, 
maxbibnames=99, 
uniquelist=false,
uniquename=init,
giveninits=true, 
natbib, 
date=year 
]{biblatex}

\DeclareFieldFormat{pages}{#1} 
\renewbibmacro{in:}{\ifentrytype{article}{}{\printtext{\bibstring{in}\intitlepunct}}} 

\usepackage[inline,shortlabels]{enumitem}
\setlist[enumerate,1]{label=(\roman*)}

\usepackage[margin = 1.25in]{geometry}
\usepackage{setspace}
\numberwithin{equation}{section}

\newcommand{\cD}{\mathcal{D}} 
\newcommand{\sA}{\mathsf{A}} 
\newcommand{\sG}{\mathsf{G}} 
\newcommand{\sV}{\mathsf{V}} 
\newcommand{\sX}{\mathsf{X}} 
\newcommand{\sZ}{\mathsf{Z}} 

\title{Unbounded Markov Dynamic Programming with Weighted Supremum Norm Perov Contractions}
\author{Alexis Akira Toda\thanks{Department of Economics, University of California San Diego. Email: \href{mailto:atoda@ucsd.edu}{atoda@ucsd.edu}.}}

\begin{document}
\maketitle
	
\begin{abstract}
	
This paper shows the usefulness of the Perov contraction theorem, which is a generalization of the classical Banach contraction theorem, for solving Markov dynamic programming problems. When the reward function is unbounded, combining an appropriate weighted supremum norm with the Perov contraction theorem yields a unique fixed point of the Bellman operator under weaker conditions than existing approaches. An application to the optimal savings problem shows that the average growth rate condition derived from the spectral radius of a certain nonnegative matrix is sufficient and almost necessary for obtaining a solution.
	
\medskip
		
\textbf{Keywords:} Dynamic programming, Gelfand formula, optimal savings, Perov contraction, spectral radius, weighted supremum norm.
		
		
\end{abstract}
	
\section{Introduction}

The classical approach to solving infinite-horizon dynamic programming problems is to show that the Bellman operator is a contraction on a space of candidate value functions and apply the Banach contraction mapping theorem to establish the existence and uniqueness of a value function satisfying the Bellman equation \citep{Shapley1953,Blackwell1965,Denardo1967}. An underlying assumption to this approach is that the reward function is bounded and hence we may consider the Banach space of bounded functions endowed with the supremum norm as the space for candidate value functions.

However, many reward functions commonly used in applications are unbounded. To deal with these situations, instead of using the supremum norm
\begin{equation*}
	\norm{v}\coloneqq \sup_{x\in \sX}\abs{v(x)},
\end{equation*}
where $\sX$ denotes the state space, one could use the weighted supremum norm defined by
\begin{equation*}
	\norm{v}_\kappa\coloneqq \sup_{x\in \sX}\frac{\abs{v(x)}}{\kappa(x)},
\end{equation*}
where $\kappa$ is some positive weight function. If the reward and value functions can be shown to be bounded above by some positive multiple of $\kappa$, we may apply the contraction approach after rescaling the reward and value functions by $\kappa$ and recover existence, uniqueness, and other optimality results. This ``weighted supremum norm'' approach was pioneered by \citet{Lippman1975,Wessels1977} and has been widely applied; see \citet{Boyd1990} and \citet{Duran2000,Duran2003} for economic applications and \citet[Ch.~8]{Hernandez-LermaLasserre1999}, \citet[Ch.~12]{Stachurski2009}, and \citet{Bertsekas2018} for textbook treatments.

One limitation of the existing weighted supremum norm approach is that the sufficient conditions are often too strong for common applications. For instance, consider Assumption 8.3.2(b) in \cite{Hernandez-LermaLasserre1999}, which can be written (with some changes in notation) as
\begin{equation}
	\beta \sup_{x\in \sX} \sup_{a\in \Gamma(x)}\frac{\E[\kappa(x') \mid x,a]}{\kappa(x)}<1, \label{eq:weight_asmp}
\end{equation}
where $\beta$ is the discount factor, $\Gamma(x)$ is the set of feasible actions given the current state $x$, and $x'$ is the next period's state given the current state and action $(x,a)$. Note that the condition \eqref{eq:weight_asmp} implies that the conditional expected growth rate of the weight function, $\E[\kappa(x') \mid x,a]/\kappa(x)$, is bounded above by $1/\beta$ uniformly over the state $x\in \sX$. Such a condition is very strong because it imposes an upper bound on the maximum growth rate of the system, which is often undesirable for particular applications.\footnote{As an illustration, consider a model in which the system switches between ``expansions'' with high growth and ``recessions'' with low growth. Conditions of the form \eqref{eq:weight_asmp} significantly restrict the maximum growth rate, which could make the model unrealistic.}

This paper seeks to relax the condition \eqref{eq:weight_asmp} within the weighted supremum norm framework. Although mathematically imprecise, roughly speaking, the uniform growth rate condition \eqref{eq:weight_asmp} is replaced with the average growth rate condition
\begin{equation}
	\beta \E\left[\sup_{a\in \Gamma(x)}\frac{\E[\kappa(x') \mid x,a]}{\kappa(x)}\right]<1, \label{eq:weight_asmp2}
\end{equation}
which is much weaker. To obtain this result, I apply a generalization of the Banach contraction theorem due to \citet{Perov1964}. While a contraction $T$ defined on a complete metric space satisfies
\begin{equation*}
	d(Tv_1,Tv_2)\le \beta d(v_1,v_2),
\end{equation*}
where $d$ is some metric and $0\le \beta<1$ is the modulus of contraction, a Perov contraction satisfies
\begin{equation*}
	d(Tv_1,Tv_2)\le Bd(v_1,v_2),
\end{equation*}
where $d$ is a vector-valued metric and $B$ is a nonnegative matrix with spectral radius $\rho(B)<1$. The Perov contraction theorem states that a Perov contraction admits a unique fixed point, and its proof is nearly identical to that of the classical Banach contraction theorem. Although the Perov contraction theorem does not seem to be well known in optimal control theory (it was initially developed to study systems of ordinary differential equations), recently \citet{Toda2021ORL} applied it to solve a dynamic programming problem with state-dependent discounting. The average growth rate condition \eqref{eq:weight_asmp2} corresponds to the spectral condition $\rho(B)<1$. In recent years, the importance of the spectral radius for solving dynamic programming problems has been recognized by several authors \citep{Toda2019JME,BorovickaStachurski2020,MaStachurskiToda2020JET}.

To illustrate the usefulness of the Perov contraction theorem coupled with weighted supremum norm for unbounded Markov dynamic programming, I apply the main results to solve an optimal savings problem with unbounded utility. I show through an example that the maximum growth rate condition \eqref{eq:weight_asmp} is restrictive and that the average growth rate condition \eqref{eq:weight_asmp2} is not only sufficient but also almost necessary.

\section{Perov contraction theorem}\label{sec:Perov}

This section introduces some notation and explains the Perov contraction theorem. For $N\in \N$, the set $\R^N$ denotes the $N$-dimensional Euclidean space with a typical element denoted by $x=(x_1,\dots,x_N)$. The set
\begin{equation*}
	\R_+^N=\set{x=(x_n)\in \R^N:(\forall n)x_n\ge 0}
\end{equation*}
denotes the nonnegative orthant. For vectors $a=(a_1,\dots,a_N)\in \R^N$ and $b=(b_1,\dots,b_N)\in \R^N$, we write $a\le b$ if and only if $a_n\le b_n$ for all $n$, or equivalently $b-a\in \R_+^N$. The set $\R^{M\times N}$ denotes the set of of all $M\times N$ real matrices.

Let $\sV$ be a set, $N\in \N$, and $d:\sV\times \sV\to \R^N$. We say that $d$ is a \emph{vector-valued metric} if the following conditions hold:
\begin{enumerate}
	\item\label{item:metric_nonnegative} (Nonnegativity) $d(v_1,v_2)\ge 0$, with equality if and only if $v_1=v_2$,
	\item\label{item:metric_symmetry} (Symmetry) $d(v_1,v_2)=d(v_2,v_1)$,
	\item\label{item:metric_triangle} (Triangle inequality) $d(v_1,v_3)\le d(v_1,v_2)+d(v_2,v_3)$.
\end{enumerate}
Note that in conditions \ref{item:metric_nonnegative} and \ref{item:metric_triangle}, inequalities are interpreted entry-wise. A set $\sV$ endowed with a vector-valued metric $d$ is called a \emph{vector-valued metric space}. Obviously, a metric space is a special case of a vector-valued metric space by setting $N=1$.

Let $\norm{\cdot}$ denote the supremum norm on $\R^N$, so $\norm{a}=\max_n\abs{a_n}$ for $a=(a_1,\dots,a_N)\in \R^N$. Note that the supremum norm satisfies the following monotonicity property: if $a,b\in \R^N$ and $0\le a\le b$, then
\begin{equation*}
	\norm{a}=\max_n a_n\le \max_n b_n=\norm{b}.
\end{equation*}
The monotonicity will be repeatedly used in the subsequent discussion. If $(\sV,d)$ is a vector-valued metric space and we define $\norm{d}:\sV\times \sV\to \R$ by
\begin{equation*}
	\norm{d}(v_1,v_2)=\norm{d(v_1,v_2)}=\max_n d_n(v_1,v_2),
\end{equation*}
then $(\sV,\norm{d})$ is a metric space in the usual sense. To see this, conditions \ref{item:metric_nonnegative} and \ref{item:metric_symmetry} are trivial, and condition \ref{item:metric_triangle} holds because
\begin{align*}
	\norm{d}(v_1,v_3)&=\norm{d(v_1,v_3)}\le \norm{d(v_1,v_2)+d(v_2,v_3)}\\
	&\le \norm{d(v_1,v_2)}+\norm{d(v_2,v_3)}=\norm{d}(v_1,v_2)+\norm{d}(v_2,v_3),
\end{align*}
where the first inequality uses condition \ref{item:metric_triangle} for $d$ and the monotonicity of the supremum norm $\norm{\cdot}$. We say that the vector-valued metric space $(\sV,d)$ is \emph{complete} if the metric space $(\sV,\norm{d})$ is complete.

Below, let $\norm{\cdot}$ also denote the operator norm for $N\times N$ matrices induced by the supremum norm, that is, $\norm{A}=\sup_{\norm{x}=1}\norm{Ax}$ for $A\in \R^{N\times N}$. Recall that for a square matrix $A$, the spectral radius $\rho(A)$ is defined by the largest absolute value of all eigenvalues:
\begin{equation*}
	\rho(A)\coloneqq \max_n \set{\abs{\alpha_n}:\text{$\alpha_n$ is an eigenvalue of $A$}}.
\end{equation*}
For any matrix norm, the Gelfand spectral radius formula
\begin{equation}
	\rho(A)=\lim_{k\to\infty} \norm{A^k}^{1/k} \label{eq:Gelfand}
\end{equation}
holds \citep[Corollary 5.6.14]{HornJohnson2013}. 

We extend the notion of contractions as follows. Let $(\sV,d)$ be a vector-valued metric space. We say that a self map $T:\sV\to \sV$ is a \emph{Perov contraction} with coefficient matrix $B\ge 0$ if $\rho(B)<1$ and
\begin{equation}
	d(Tv_1,Tv_2)\le Bd(v_1,v_2) \label{eq:metric_Perov}
\end{equation}
for all $v_1,v_2\in \sV$. Here $B\ge 0$ means that the matrix $B=(b_{mn})$ is nonnegative: $b_{mn}\ge 0$ for all $m,n$. When $T$ is a Perov contraction, by iterating \eqref{eq:metric_Perov}, for every $k\in \N$ we have
\begin{equation*}
	d(T^kv_1,T^kv_2)\le B^kd(v_1,v_2).
\end{equation*}
Taking the supremum norm of both sides, we obtain
\begin{equation*}
	\norm{d}(T^kv_1,T^kv_2)\le \norm{B^k}\norm{d}(v_1,v_2).
\end{equation*}
Noting the Gelfand spectral radius formula \eqref{eq:Gelfand}, it follows that $T^k$ is a contraction if $k$ is large enough. Thus a Perov contraction is nothing but an eventual contraction ($k$-stage contraction for some $k\in \N$). The following fixed point theorem is therefore not surprising.

\begin{thm}[Perov contraction theorem \citep{Perov1964}]\label{thm:metric_Perov}
	Let $(\sV,d)$ be a complete vector-valued metric space and $T:\sV\to \sV$ be a Perov contraction with coefficient matrix $B\ge 0$. Then
	\begin{enumerate}
		\item $T$ has a unique fixed point $v^*\in \sV$,
		\item for any $v_0\in \sV$, we have $v^*=\lim_{k\to\infty}T^kv_0$, and
		\item for any $\beta\in (\rho(B),1)$, the approximation error $d(T^kv_0,v^*)$ has order of magnitude $\beta^k$.
	\end{enumerate}
\end{thm}

\begin{proof}
The proof is nearly identical to that of the classical contraction mapping theorem except that the monotonicity of the supremum norm and the Gelfand spectral radius formula play important roles. See \cite{Toda2021ORL} for details.
\end{proof}

The following proposition generalizes \citet{Blackwell1965}'s sufficient condition to Perov contractions.

\begin{prop}\label{prop:metric_PerovBlackwell}
Let $X$ be a set and $\sV$ be a space of functions $v:X\to \R^N$ with the following properties:
\begin{enumerate}[(a)]
	\item (Upward shift) For $v\in \sV$ and $c\in \R_+^N$, we have $v+c\in \sV$.
	\item (Bounded difference) For all $u,v\in \sV$ and $n$, we have
	\begin{equation*}
		d_n(u,v)\coloneqq \sup_{x\in X}\abs{u_n(x)-v_n(x)}<\infty.
	\end{equation*}
\end{enumerate}
Let $d=(d_1,\dots,d_N)$. Suppose that $(\sV,d)$ is a complete vector-valued metric space and $T:\sV \to \sV$ satisfies
\begin{enumerate}
	\item (Monotonicity) $u\le v$ implies $Tu\le Tv$,
	\item (Discounting) there exists a nonnegative matrix $B\in \R_+^{N\times N}$ with $\rho(B)<1$ such that, for all $v\in \sV$ and $c\in \R_+^N$, we have $T(v+c)\le Tv+Bc$.
\end{enumerate}
Then $T$ is a Perov contraction with coefficient matrix $B$.
\end{prop}

\begin{proof}
Take any $v_1,v_2\in \sV$ and let $c=d(v_1,v_2)\in \R_+^N$. For any $x\in X$, we have
\begin{equation*}
	v_1(x)=v_1(x)-v_2(x)+v_2(x)\le v_2(x)+c,
\end{equation*}
so $v_1\le v_2+c\in \sV$ by the upward shift property. Using monotonicity and discounting, we obtain
\begin{equation*}
	Tv_1\le T(v_2+c)\le Tv_2+Bc\implies Tv_1-Tv_2\le Bc.
\end{equation*}
Interchanging the role of $v_1,v_2$, we obtain $Tv_2-Tv_1\le Bc$. Combining these two inequalities, for every $m=1,\dots,N$ and $x\in X$, we have
\begin{equation*}
	\abs{(Tv_1)_m(x)-(Tv_2)_m(x)}\le (Bd(v_1,v_2))_m.
\end{equation*}
Taking the supremum over $x\in X$ and using the definition of the vector-valued metric $d$, we obtain $d(Tv_1,Tv_2)\le Bd(v_1,v_2)$, so $T$ is a Perov contraction with coefficient matrix $B$.
\end{proof}

\begin{rem*}
A version of Proposition \ref{prop:metric_PerovBlackwell} appears in \citep[Theorem 3]{Toda2021ORL}. 
Proposition A.4 of \cite{StachurskiZhang2021} and Proposition 6.3.3 of \cite{SargentStachurskiDPbook} obtain a similar result in the context of eventual contractions.
\end{rem*}

\section{Unbounded Markov dynamic programming}

In this section, we apply the Perov contraction theorem within the weighted supremum norm framework to solve unbounded Markov dynamic programming problems.

\subsection{Abstract dynamic program}

We first introduce the notion of an abstract dynamic program following \cite{Denardo1967,Bertsekas2018}. A \emph{dynamic program} is a tuple $\cD=\set{\sX,\sA,\Gamma,\sV,H}$, where
\begin{itemize}
	\item $\sX$ is a nonempty set called the \emph{state space},
	\item $\sA$ is a nonempty set called the \emph{action space},
	\item $\Gamma:\sX\twoheadrightarrow \sA$ is a nonempty correspondence called the \emph{feasible correspondence}, with its graph denoted by $\sG\coloneqq \set{(x,a)\in \sX\times \sA: a\in \Gamma(x)}$,
	\item $\sV$ is a nonempty space of functions $v:\sX\to [-\infty,\infty]$ called the \emph{value space},
	\item $H:\sG\times \sV\to [-\infty,\infty]$ is a function called the \emph{aggregator}, which is increasing in the last argument:
	\begin{equation*}
		v_1\le v_2\implies H(x,a,v_1)\le H(x,a,v_2).
	\end{equation*}
\end{itemize}

Given a dynamic program $\cD=\set{\sX,\sA,\Gamma,\sV,H}$, the Bellman operator $T$ on the value space $\sV$ is defined by
\begin{equation}
	(Tv)(x)\coloneqq \sup_{a\in \Gamma(x)}H(x,a,v), \label{eq:dpintro_T}
\end{equation}
where $v\in \sV$ and $x\in \sX$. We say that $v\in \sV$ is a \emph{value function} if $v$ is a fixed point of the Bellman operator $T$, that is, $Tv=v$. In what follows, we introduce the following additional structure, which we refer to as an \emph{additive Markov dynamic program}.
\begin{itemize}
	\item The state space can be written as $\sX\times\sZ$, where $\sZ=\set{1,\dots,Z}$ is a finite set associated with a stochastic matrix $P=(P(z,z'))_{z,z'\in \sZ}$.\footnote{We say $P=(P(z,z'))_{z,z'\in \sZ}$ is a stochastic matrix if $P(z,z')\ge 0$ and $\sum_{z'\in \sZ}P(z,z')=1$.}
	\item The aggregator takes the additive (expected utility) form
	\begin{equation}
		H(x,z,a,v)=r(x,z,a)+\sum_{z'=1}^ZP(z,z')\beta(z,z')v(g(x,z,z',a),z'), \label{eq:contract_HMDP}
	\end{equation}
	where $r:\sX\times \sZ\times \sA\to [-\infty,\infty)$ is the \emph{reward function}, $g:\sX\times \sZ^2\times \sA \to \sX$ is the \emph{law of motion} or \emph{transition function}, and $\beta(z,z')\ge 0$ is the \emph{discount factor} conditional on transitioning from state $z$ to $z'$.
\end{itemize}
Some remarks are in order. First, the additivity of the aggregator $H$ in $v$ implies that we focus on expected utility models. Such a restriction is not essential but simplifies the subsequent discussion. For some approaches that do not require additivity, see for example \cite{Duran2000,Duran2003}. Second, we allow the discount factor $\beta(z,z')$ to be state dependent. Obviously, the classical setting in which the discount factor $\beta(z,z')=\beta$ is constant is a special case. State-dependent discounting causes no theoretical difficulty but has been shown to be an important economic feature; see for instance \cite{StachurskiZhang2021,Toda2021ORL} and the references therein. Third, for simplicity we suppose that the uncertainty is driven by an exogenous finite-state Markov chain. The generalization to infinite state spaces should be possible by applying generalizations of the Perov contraction theorem reviewed in \cite{Zabreiko1997}.

Note that in the definition of the aggregator \eqref{eq:contract_HMDP}, the summation can be interpreted as the conditional expectation
\begin{equation*}
	\E[\beta(z_t,z_{t+1})v(x_{t+1},z_{t+1}) \mid z_t=z],
\end{equation*}
where the next state is
\begin{equation*}
	x_{t+1}=g(x_t,z_t,z_{t+1},a_t).
\end{equation*}
Thus we may rewrite the Bellman operator $T$ in \eqref{eq:dpintro_T} as
\begin{align}
	(Tv)(x,z)&\coloneqq \sup_{a\in \Gamma(x,z)} H(x,z,a,v)\notag \\
	&=\sup_{a\in \Gamma(x,z)}\set{r(x,z,a)+\E_z[\beta(z,z')v(x',z')]}, \label{eq:contract_TMDP}
\end{align}
where $\E_z=\E[\cdot \mid z]$ denotes the conditional expectation and it is understood that $x'=g(x,z,z',a)$. We write an additive Markov dynamic program as
\begin{equation}
	\cD=\set{\sX,\sZ,P,\sA,\Gamma,\sV,r,g,\beta}. \label{eq:contract_additiveMDP}
\end{equation}

\subsection{Weighted supremum norm approach}

Let $\cD$ be an additive Markov dynamic program in \eqref{eq:contract_additiveMDP}. If the reward function $r$ is bounded and the discount factor $\beta(z,z')$ is less than 1, it is well known that the Bellman operator $T$ in \eqref{eq:contract_TMDP} is a contraction on the space of bounded functions equipped with the supremum norm and hence a unique fixed point exists. However, in many common applications, the reward function is unbounded. One way to get around this issue is to use a weighted supremum norm instead of the (unweighted) supremum norm as proposed by \cite{Lippman1975,Wessels1977}.

We start the discussion with the Bellman equation
\begin{equation}
	v(x,z)=\sup_{a\in \Gamma(x,z)}\set{r(x,z,a)+\E_z[\beta(z,z')v(x',z')]}, \label{eq:contract_bellmanMDP}
\end{equation}
which corresponds to setting $Tv=v$ in \eqref{eq:contract_TMDP}. Let $\kappa:\sX\times \sZ\to (0,\infty)$ be some positive function and suppose that we normalize the value function as $\tilde{v}=v/\kappa$. Then the Bellman equation \eqref{eq:contract_bellmanMDP} becomes
\begin{equation*}
	\kappa(x,z)\tilde{v}(x,z)=\sup_{a\in \Gamma(x,z)}\set{r(x,z,a)+\E_z[\beta(z,z')\kappa(x',z')\tilde{v}(x',z')]}.
\end{equation*}
Dividing both sides by $\kappa(x,z)>0$, we may define the scaled Bellman operator $\tilde{T}$ by
\begin{equation}
	(\tilde{T}\tilde{v})(x,z)=\sup_{a\in \Gamma(x,z)}\set{\tilde{r}(x,z,a)+\E_z\left[\beta(z,z')\frac{\kappa(x',z')}{\kappa(x,z)}\tilde{v}(x',z')\right]}, \label{eq:contract_bellmanweight}
\end{equation}
where $\tilde{r}\coloneqq r/\kappa$. To make $\tilde{T}$ a (Perov) contraction, all we need is to control the ratio $\kappa(x',z')/\kappa(x,z)$. We thus define
\begin{equation}
	\tilde{\beta}(z,z')\coloneqq \beta(z,z')\sup_{x\in \sX}\sup_{a\in \Gamma(x,z)}\frac{\kappa(g(x,z,z',a),z')}{\kappa(x,z)}. \label{eq:contract_gamma}
\end{equation}
To come up with the appropriate function space, let $\sV$ be the space of functions $v:\sX\times \sZ\to \R$ with
\begin{equation}
	\norm{v}_\kappa \coloneqq \sup_{z\in \sZ}\sup_{x\in \sX}\frac{\abs{v(x,z)}}{\kappa(x,z)}<\infty. \label{eq:contract_psinorm}
\end{equation}
Because $\kappa>0$, it is straightforward to show that $(\sV,\norm{\cdot}_\kappa)$ is a Banach space. The norm \eqref{eq:contract_psinorm} is called the \emph{weighted supremum norm} with weight function $\kappa$. For $v_1,v_2\in \sV$, if we define $d:\sV\times \sV \to \R_+^Z$ by
\begin{equation*}
	d_z(v_1,v_2)=\sup_{x\in \sX} \frac{\abs{v_1(x,z)-v_2(x,z)}}{\kappa(x,z)},
\end{equation*}
then $(\sV,d)$ becomes a complete vector-valued metric space by the discussion in Section \ref{sec:Perov}. In what follows, it is convenient to define the space $(b\sX)^Z$ of functions $f=(f_1,\dots,f_Z):\sX\to \R^Z$, where each $f_z$ is bounded on $\sX$. Obviously, $(b\sX)^Z$ is a complete vector-valued metric space by considering the sup distance for each entry.

With this preparation, we obtain the following theorem, which is the main result of this paper.

\begin{thm}\label{thm:contract_weight}
Let $\cD$ in \eqref{eq:contract_additiveMDP} be an additive Markov dynamic program associated with a weight function $\kappa:\sX\times \sZ\to (0,\infty)$. Let $(\sV,d)$ be the complete vector-valued metric space just described. Suppose that
\begin{equation}
	\sup_{x\in \sX}\sup_{a\in \Gamma(x,z)}\frac{\abs{r(x,z,a)}}{\kappa(x,z)}<\infty
\end{equation}
and $\rho(B)<1$, where the nonnegative matrix $B\coloneqq (P(z,z')\tilde{\beta}(z,z'))$ is defined using \eqref{eq:contract_gamma}. Then the following statements are true.
\begin{enumerate}
	\item\label{item:contract_weight1} The (scaled) Bellman operator $T$ ($\tilde{T}$) is a Perov contraction on $\sV$ ($(b\sX)^Z$) with coefficient matrix $B$.
	\item\label{item:contract_weight2} $\cD$ has a unique value function $v=\kappa \tilde{v}$ in $\sV$, where $\tilde{v}$ is the unique fixed point of $\tilde{T}$ in $(b\sX)^Z$.
\end{enumerate}
\end{thm}

\begin{proof}
\ref{item:contract_weight1} It suffices to show the claim for $\tilde{T}$. We verify the assumptions of Proposition \ref{prop:metric_PerovBlackwell}. It is clear that $(b\sX)^Z$ satisfies the upward shift and bounded difference properties. The monotonicity of $\tilde{T}$ immediately follows from the definition \eqref{eq:contract_bellmanweight}. To show discounting, take any $c\in \R_+^Z$. Using \eqref{eq:contract_bellmanweight} and \eqref{eq:contract_gamma}, we obtain
\begin{align*}
	&(\tilde{T}(\tilde{v}+c))(x,z)\\
	&=\sup_{a\in \Gamma(x,z)}\set{\tilde{r}(x,z,a)+\E_z\left[\beta(z,z')\frac{\kappa(x',z')}{\kappa(x,z)}(\tilde{v}(x',z')+c(z'))\right]}\\
	&\le (\tilde{T}\tilde{v})(x,z)+\sup_{a\in \Gamma(x,z)}\E_z\left[\beta(z,z')\frac{\kappa(x',z')}{\kappa(x,z)}c(z')\right]\\
	&\le (\tilde{T}\tilde{v})(x,z)+\E_z[\tilde{\beta}(z,z')c(z')]\\
	&=(\tilde{T}\tilde{v})(x,z)+(Bc)_z.
\end{align*}
Therefore $\tilde{T}(\tilde{v}+c)\le \tilde{T}\tilde{v}+Bc$, so discounting holds.
	
\ref{item:contract_weight2} Obvious by \eqref{eq:contract_bellmanMDP} and \eqref{eq:contract_bellmanweight}.
\end{proof}

\subsection{Discussion}

Although the proof of Theorem \ref{thm:contract_weight} is a straightforward application of the Perov contraction theorem \ref{thm:metric_Perov} and the Blackwell-type sufficient condition (Proposition \ref{prop:metric_PerovBlackwell}), the value of Theorem \ref{thm:contract_weight} relative to existing results is the weakness of the assumption and the simplicity of the argument. Regarding the assumption, the existing literature typically assumes that the quantity $\tilde{\beta}(z,z')$ in \eqref{eq:contract_gamma} (or $\sum_{z'=1}^ZP(z,z')\tilde{\beta}(z,z')$) is uniformly bounded above by 1 in order to apply the contraction mapping theorem directly; for a textbook treatment, see for instance \citet[Assumption 8.3.2(b)]{Hernandez-LermaLasserre1999}.\footnote{Similar assumptions appear in \cite[Assumption 1]{Lippman1975}, \cite[Assumption (2)]{Wessels1977}, \cite[Assumption (A2)]{Duran2000},  \cite[Assumption 4]{Duran2003}, \cite[Assumption 12.2.14]{Stachurski2009}, \cite[Assumption 2.1.2]{Bertsekas2018}, and \cite[Assumption 5.1]{MaStachurskiToda2022JME}, among others.} Such uniform boundedness assumption is sufficient but not necessary for the contraction argument. Furthermore, as we shall see in the example below, such an assumption is too restrictive for applications.

Regarding the simplicity of the argument, as discussed before, because a Perov contraction is nothing but an eventual contraction, no new mathematical results are necessary to obtain Theorem \ref{thm:contract_weight}. In fact, several authors directly prove that the operator is an eventual contraction. See, for example, \cite[Lemma B.5]{MaStachurskiToda2020JET}. However, a significant advantage of the Perov contraction approach is that the proofs become very clear.

\section{Application: optimal savings}

To illustrate the power of Theorem \ref{thm:contract_weight} as well as the limitation of existing results, we consider the following optimal savings problem:
\begin{subequations}\label{eq:contract_os}
\begin{align}
	&\maximize && \E_0\sum_{t=0}^\infty \beta^t u(c_t) \label{eq:contract_osobj}\\
	&\st && (\forall t) w_{t+1}=R(z_t,z_{t+1})(w_t-c_t)+y(z_{t+1}) \label{eq:contract_osbudget}\\
	& && (\forall t) 0\le c_t\le w_t, \label{eq:contract_osborrow}\\
	&&& \text{$w_0>0$, $z_0$ given.} \label{eq:contract_osinitial}
\end{align}
\end{subequations}
Here $u:\R_+\to [-\infty,\infty)$ is the flow utility function from consumption $c_t\ge 0$ at time $t$; the parameter $\beta\in [0,1)$ is the discount factor; $\E_t$ denotes the expectation conditional on time $t$ information; $w_t\ge 0$ is the financial wealth at the beginning of time $t$; $\set{z_t}_{t=0}^\infty$ is a Markov chain taking values in the finite set $\sZ=\set{1,\dots,Z}$ with transition probability matrix $P=(P(z,z'))$; $y:\sZ\to \R_+$ specifies the non-financial income of the agent in each state $z\in \sZ$; and $R:\sZ^2\to \R_+$ specifies the gross return\index{gross return} on savings conditional on transitioning from state $z$ to $z'$. The expression \eqref{eq:contract_osobj} is the objective function; the condition \eqref{eq:contract_osbudget} is the budget constraint; the condition \eqref{eq:contract_osborrow} implies that consumption is nonnegative and the agent cannot borrow; and \eqref{eq:contract_osinitial} is the initial condition. To understand the budget constraint \eqref{eq:contract_osbudget}, note that the next period's financial wealth $w'$ is the sum of the next period's non-financial income $y'$ and the return from savings, which is $R$ times $w-c$.

The optimal savings problem \eqref{eq:contract_os} is an important building block of many economic models and has been studied under various specifications and assumptions. While \cite{Schechtman1976,ChamberlainWilson2000} assume a bounded utility function and apply the contraction mapping theorem to the Bellman equation, \cite{LiStachurski2014,MaStachurskiToda2020JET,MaToda2021JET,MaToda2022JME} do away with boundedness and apply a contraction argument to the Euler equation (first-order optimality condition).

Here we solve the optimal savings problem when $u$ could be unbounded by applying the Perov contraction theorem and the weighted supremum norm. Because constant discounting is inessential, let $\beta(z,z')\ge 0$ be the discount factor conditional on transitioning from state $z$ to $z'$. Then the Bellman equation becomes
\begin{equation}
	v(w,z)=\sup_{0\le c\le w}\set{u(c)+\E_z[\beta(z,z')v(w',z')]}, \label{eq:os_bellman}
\end{equation}
where the next period's wealth is
\begin{equation*}
	w'=g(w,z,z',c)\coloneqq R(z,z')(w-c)+y(z').
\end{equation*}

We impose the following assumption.

\begin{asmp}\label{asmp:U}
The utility function $u:[0,\infty)\to \R$ is increasing, concave, and bounded below.
\end{asmp}

A typical example satisfying Assumption \ref{asmp:U} is the constant relative risk aversion (CRRA) specification
\begin{equation}
	u(c)=\frac{c^{1-\gamma}}{1-\gamma} \label{eq:contract_CRRA}
\end{equation}
with $0<\gamma<1$. Note that $u$ in \eqref{eq:contract_CRRA} is unbounded above.

The following proposition shows the existence and uniqueness of a value function.

\begin{prop}\label{prop:os}
Consider the optimal savings problem \eqref{eq:contract_os} and suppose Assumption \ref{asmp:U} holds. Define the matrix $B\in \R_+^{Z\times Z}$ by
\begin{equation}
	B(z,z')=P(z,z')\beta(z,z')\max\set{1,R(z,z')}. \label{eq:Bzz'}
\end{equation}
If $\rho(B)<1$, then the Bellman equation \eqref{eq:os_bellman} has a unique value function satisfying
\begin{equation*}
	\sup_{z\in \sZ}\sup_{w\ge 0}\frac{\abs{v(w,z)}}{w+1}<\infty.
\end{equation*}
\end{prop}

\begin{proof}
Since $u$ is increasing and bounded below, we have $u(c)\ge u(0)>-\infty$. By redefining $u(c)$ as $u(c)-u(0)$ (which is a monotonic transformation that does not affect preference ordering) if necessary, without loss of generality we may assume $u\ge 0$. Since $u$ is increasing and concave, we can take $a,b>0$ such that $u(c)\le ac+b$ for all $c\ge 0$. By redefining $u(c)$ as $u(c)/a$ if necessary, without loss of generality we may assume $a=1$. Therefore $0\le u(c)\le c+b$, and clearly we can take arbitrarily large $b>0$.

Consider the weight function $\kappa(w,z)=w+b$, where $b>0$. For $0\le c\le w$, the normalized utility is
\begin{equation*}
	0\le \tilde{u}(w,c,z)\coloneqq \frac{u(c)}{\kappa(w,z)}\le \frac{u(w)}{w+b}\le 1,
\end{equation*}
which is bounded. Furthermore,
\begin{align}
	\frac{\kappa(g(w,z,z',c),z')}{\kappa(w,z)}&=\frac{R(z,z')(w-c)+y(z')+b}{w+b}\notag \\
	&\le \frac{R(z,z')w+y(z')+b}{w+b}. \label{eq:contract_osub}
\end{align}
Noting that
\begin{equation*}
	\frac{Rw+y+b}{w+b}\le\frac{\max\set{1,R}(w+b)+y}{w+b}\le \max\set{1,R}+\frac{y}{b},
\end{equation*}
it follows from \eqref{eq:contract_osub} that
\begin{equation*}
	\frac{\kappa(g(w,z,z',c),z')}{\kappa(w,z)}\le \max\set{1,R(z,z')}+\frac{y(z')}{b}\to \max\set{1,R(z,z')}
\end{equation*}
as $b\to\infty$. Therefore by taking $b$ large enough, a sufficient condition for the existence and uniqueness of a fixed point is that $\tilde{\beta}(z,z')\coloneqq \beta(z,z') \max\set{1,R(z,z')}$ satisfies the assumption of Theorem \ref{thm:contract_weight}, which is $\rho(B)<1$ for $B$ defined by \eqref{eq:Bzz'}.
\end{proof}

\begin{rem*}
Proposition \ref{prop:os} illustrates the limitation of the existing weighted supremum norm approach, which requires the condition
\begin{equation}
	\E_z[\beta(z,z')\max\set{1,R(z,z')}]=\sum_{z'=1}^ZP(z,z')\beta(z,z')\max\set{1,R(z,z')}<1 \label{eq:weight_cond}
\end{equation}
for all $z$ instead of $\rho(B)<1$. Since \eqref{eq:weight_cond} is equivalent to $B1\ll 1$ (where $1=(1,\dots,1)$ denotes the vector of ones), the well-known property of the spectral radius \citep[Theorem 8.1.22]{HornJohnson2013} shows that \eqref{eq:weight_cond} implies $\rho(B)<1$, so the former condition is stronger. For instance, suppose one period in the model corresponds to a year and the agent discounts future utility at 5\%. Then $\beta=0.95$. In order for \eqref{eq:weight_cond} to hold, the conditional expected gross return on wealth $\E_z[R(z,z')]$ cannot exceed 1.05, but the expected return on common assets such as stocks can easily exceed this value. In contrast, the condition $\rho(B)<1$ only requires that $\beta R<1$ in the long run average, not conditionally.
\end{rem*}

\begin{rem*}
In the setting of Proposition \ref{prop:os}, suppose that the utility function is given by \eqref{eq:contract_CRRA} with $0<\gamma<1$. If we consider the weight function $\kappa(w,z)=(w+b)^{1-\gamma}$, by a similar argument we may set
\begin{equation*}
	B(z,z')\coloneqq P(z,z')\beta(z,z') \max\set{1,R(z,z')^{1-\gamma}},
\end{equation*}
and satisfying the assumptions of Theorem \ref{thm:contract_weight} becomes even easier (because $R^{1-\gamma}<R$ whenever $R>1$).
\end{rem*}

In Theorem \ref{thm:contract_weight}, the condition $\rho(B)<1$ is sufficient for applying the Perov contraction theorem to solve an unbounded Markov dynamic program. The following proposition shows that this condition is almost necessary, which implies that it would be difficult to improve Theorem \ref{thm:contract_weight}.

\begin{prop}\label{prop:os_CRRA}
Consider the optimal savings problem \eqref{eq:contract_os} with CRRA utility function \eqref{eq:contract_CRRA} with $0<\gamma<1$ and zero non-financial income: $y(z)=0$ for all $z\in \sZ$. Define the matrix $B\in \R_+^{Z\times Z}$ by
\begin{equation}
	B(z,z')=P(z,z')\beta(z,z')R(z,z')^{1-\gamma}. \label{eq:Bzz'_CRRA}
\end{equation}
Then the following statements are true.
\begin{enumerate}
	\item\label{item:os_CRRA1} If $\rho(B)<1$, the Bellman equation \eqref{eq:os_bellman} has a unique value function satisfying
	\begin{equation*}
		\sup_{z\in \sZ}\sup_{w>0}\frac{\abs{v(w,z)}}{w^{1-\gamma}}<\infty.
	\end{equation*}
	\item\label{item:os_CRRA2} If $\rho(B)>1$, the optimal value of the problem \eqref{eq:contract_os} is $\infty$.
\end{enumerate}
\end{prop}

\begin{proof}
\ref{item:os_CRRA1} Consider the weight function $\kappa(w,z)=w^{1-\gamma}$. For $0\le c\le w$, the normalized utility is
\begin{equation*}
	0\le \tilde{u}(w,c,z)\coloneqq \frac{u(c)}{\kappa(w,z)}=\frac{1}{1-\gamma}(c/w)^{1-\gamma}\le \frac{1}{1-\gamma},
\end{equation*}
which is bounded. Furthermore,
\begin{equation*}
	\frac{\kappa(g(w,z,z',c),z')}{\kappa(w,z)}=\frac{(R(z,z')(w-c))^{1-\gamma}}{w^{1-\gamma}}\le R(z,z')^{1-\gamma},
\end{equation*}
with equality if $c=0$. Therefore if $B\in \R_+^{Z\times Z}$ defined by \eqref{eq:Bzz'_CRRA} satisfies $\rho(B)<1$, by Theorem \ref{thm:contract_weight} there exists a unique fixed point of the Bellman operator.

\ref{item:os_CRRA2} Consider the following feasible plan: the agent consumes zero (and saves everything) up to time $t=T$, consumes all wealth at $t=T$, and then consumes zero thereafter. Let $v_T(w,z)$ be the lifetime utility associated with this plan when the initial wealth and state are $w$ and $z$. Then clearly
\begin{equation}
	v_T(w,z)=\begin{cases*}
		\frac{w^{1-\gamma}}{1-\gamma} & if $T=0$,\\
		\E_z[\beta(z,z')v_{T-1}(R(z,z')w,z')] & if $T\ge 1$.
	\end{cases*}
	\label{eq:vT}
\end{equation}
Define the vector $a_T\in \R_+^Z$ by $a_0=1=(1,\dots,1)$ and $a_T=Ba_{T-1}$, so $a_T=B^T1$. Using \eqref{eq:vT} and the definition of $B$ in \eqref{eq:Bzz'_CRRA}, it is straightforward to show by induction that
\begin{equation}
	v_T(w,z)=a_T(z)\frac{w^{1-\gamma}}{1-\gamma}=(e_z'B^T1)\frac{w^{1-\gamma}}{1-\gamma},\label{eq:vTwz}
\end{equation}
where $e_z=(0,\dots,1,\dots,0)$ denotes the $z$-th unit vector in $\R^Z$. Since $B\ge 0$, $e_z\ge 0$, and $1\ge 0$, \eqref{eq:vTwz} can be regarded as a vector norm of $B^T$. Therefore by the generalization of the Gelfand formula \citep[Theorem 5.7.10]{HornJohnson2013}, we obtain
\begin{equation*}
	\lim_{T\to\infty} (v_T(w,z))^{1/T}=\rho(B).
\end{equation*}
Therefore if $\rho(B)>1$, then $v_T(w,z)\to\infty$ as $T\to\infty$, so the optimal value of the problem \eqref{eq:contract_os} is $\infty$.
\end{proof}
	
\printbibliography
	
\end{document}